\documentclass[a4paper,amsmath,amscd,amsbsy,amssymb]{amsart}
\newtheorem{thm}{Theorem}[section]

\newtheorem{lem}[thm]{Lemma}
\newtheorem{prop}[thm]{Proposition}

\theoremstyle{definition}

\newtheorem{que}[thm]{Question}
\theoremstyle{remark}

\numberwithin{equation}{section}
\newcommand{\R}{\mathbb R}
\newcommand{\Z}{\mathbb Z}
\newcommand{\N}{\mathbb N}
\DeclareMathSymbol{\emptyset}{\mathord}{AMSb}{"3F}
\DeclareMathSymbol{\preccurlyeq}{\mathrel}{AMSa}{"34}

\newcommand{\lpm}{{\mathcal P}{\mathcal M}_{{\mathrm L}{\mathrm i} {\mathrm p}}}
\newcommand{\lm}{{\mathcal M}_{{\mathrm L}{\mathrm i} {\mathrm p}}}
\newcommand{\lpu}{{\mathcal U}{\mathcal M}_{{\mathrm L}{\mathrm i} {\mathrm p}}}
\newcommand{\lpf}{{\mathcal C}_{{\mathrm L}{\mathrm i} {\mathrm p}}}
\usepackage[all]{xy}
\begin{document}  

\title[Extending Lipschitz (pseudo)metric]{\bf Nonexistence of linear operators extending Lipschitz (pseudo)metric}

\author[D. Repov\v{s}]{DU\v{S}AN REPOV\v{S} (Ljubljana)}
\address{Faculty of Mathematics and Physics, and Faculty of Education,
University of Ljubljana,
P.O.B. 2964, 1001 Ljubljana,
Slovenia}
\email{dusan.repovs@guest.arnes.si}

\author[M. Zarichnyi]{MICHAEL ZARICHNYI (Lviv and Rzesz\'ow)}
\address{Department of Mechanics and Mathematics,
Lviv National University,
Universytetska Str. 1,
79000 Lviv, 
Ukraine\\
and\\
Institute of Mathematics,
University of Rzesz\'ow,
Rejtana 16 A,
35-310 Rze\-sz\'ow,
Poland}
\email{mzar@litech.lviv.ua}

\thanks{This research was supported by the
Slovenian Research Agency grants P1-0292-0101 and J1-2057-0101.
The authors are indebted to Ihor Stasyuk for pointing an error in the preliminary version. We also thank all  referees for their remarks.}




\begin{abstract} We present
an example of a zero-dimensional compact metric
space $X$ and its closed subspace $A$ such that there is no continuous linear
extension operator for the Lipschitz pseudometrics on $A$ to the Lipschitz
pseudometrics on $X$. The construction is based on results of A. Brudnyi and Yu.
Brudnyi concerning linear extension operators for Lipschitz functions.
\end{abstract}

\maketitle

\section{Introduction}

The problem of extensions of metric has a long history. Hausdorff was the 
first to prove that any continuous metric defined on a closed subset of a
metrizable space can be extended to a continuous metric defined on the whole
space. Bessaga  \cite{be2,be1} considered the problem of existence of linear extension
operators for metrics and provided a partial
solution of this problem. The problem was completely solved by Banakh \cite{b}
(see also a short proof in \cite{z}).

The problem of extending Lipschitz and uniform (pseudo)metrics was
considered in \cite{L}. It is well-known that any Lipschitz pseudometric defined on a
closed subset of a metric space admits an extension which is a Lipschitz
pseudometric defined on the whole space. Recently, 
an extension operator for Lipschitz and uniform pseudometrics
was constructed in \cite{bbst}
which preserves the Lipschitz norm and has many nice properties.

In the present paper we shall consider the
problem of
existence of linear extension operators for Lipschitz pseudometric. To the
best of our
knowledge, no affirmative results are known
in this direction. Our
goal
is to construct a counterexample, i.e. show that
there exists a subset of a
zero-dimensional compact metric space for which there is no such 
extension
operator (recall that a space is zero-dimensional if its topology  possesses
a base consisting of sets which are open and closed). The example is based
on results from \cite{BB} concerning the
linear extension operators for Lipschitz functions.

Note that conditions for existence of extensions of continuous functions are
often equivalent to those for existence of extensions of continuous pseudometrics
(cf. e.g. \cite{ASh,A,Sh,Se}). It turns out that in
the case of linear extension operators, one is able to procede at least in one
direction, namely from pseudometrics to functions.

\section{Preliminaries}

Let $(X,\varrho)$ be a compact metric space. 
Given a subset $A$ of $X$,  a
pseudometric $d$ on $A$ is said to be
{\it Lipschitz} if there exists $C>0$
such that
$$d(x,y)\le C\varrho (x,y) \ \ \hbox{\rm for any} \ \ x,y\in A.$$ 

Next, a map
$f\colon X\to Y$, where $(Y,d)$ is a metric space, is called
$C$-{\it Lipschitz}, $C>0$, if
$$d(f(x),f(y))\le C\varrho(x,y) \ \ \hbox{\rm for every} \ \  x,y\in X.$$

If the knowledge of $C$ is not important, we  shall use the term ``Lipschitz map'' instead of ``$C$-Lipschitz map''.

The set of reals $\R$ is endowed with the standard metric $|x-y|$.
A bijective map $f$ of metric spaces is said to be 
{\em bi-Lipschitz} if both
$f$ and $f^{-1}$ are Lipschitz. Metric spaces are called 
{\em bi-Lipschitz equivalent}
if there exists
some bi-Lipschitz map between them.

Until the end of this section, $A$ will denote a nonempty closed subset in a compact metric space $(X,\varrho)$.
Denote by $\lpm(A)$ (resp.
$\lm(A)$, $\lpf(A)$) the set of all Lipschitz pseudometrics (resp.
metrics, functions) on $A$. The set $\lpm(A)$ (resp. $\lpf(A)$) is a cone
(resp. linear space) with respect to the operations of pointwise addition
and multiplication by scalar.

We endow $\lpm(A)$ with the norm $\|\cdot\|_{\lpm(A)}$ which is defined as follows:

$$\|d\|_{\lpm(A)}=\sup\left\{\frac{d(x,y)}{\varrho(x,y)}\mid x\neq y\right\}$$

and we put on $\lpf(A)$  the seminorm $\|\cdot\|_{\lpf(A)}$ which is defined as follows:

$$\|f\|_{\lpf(A)}=\sup\left\{\frac{|f(x)-f(y)|}{\varrho(x,y)}\mid x\neq y\right\}$$
(in the sequel, we shall
abbreviate $\|\cdot\|_{\lpm(A)}$ and $\|\cdot\|_{\lpf(A)}$ by
$\|\cdot\|_A$).

We say that a map $u\colon \lpm(A)\to\lpm(X)$ is an {\em extension operator} for
Lipschitz pseudometrics if the following holds:
\begin{enumerate}
  \item $u$ is linear, i.e.
  $$u(d_1+d_2)=u(d_1)+u(d_2),
  u(\lambda d)=\lambda
  u(d)
  \ \ \hbox{\rm for every} \ \
  d_1,d_2\in \lpm(A), \lambda\in\R_+);$$

  \item $u(d)|(A\times A)=d$, for every $d\in\lpm(A)$; and

  \item $u$ is continuous in the sense that 
  $$\|u\|=\sup\{\|u(d)\|_X\mid \|d\|_A\le1\}
  \ \ \hbox{\rm   is finite}.$$
\end{enumerate}

This definition is a natural analogue
of that
introduced in \cite{BB} for the
extensions of Lipschitz functions. The following notation was introduced in
\cite{BB}: $$\lambda(A,X)=\inf\{\|u\|\mid u\text{ is a linear extension operator
from }\lpf(A)\text{ to }\lpf(X)\}.$$

Similarly, we define
\begin{align*}
 \Lambda(A,X)= & \inf\{\|u\|\mid u\text{ is a
linear extension operator} \\
   & \lpm(A)\to \lpm(X)\}.
\end{align*}

It can be easily proved  that $$\Lambda(X)=\sup\{\Lambda(A,X)\mid
A\subset X\}$$ is a bi-Lipschitz {\it invariant} of a metric space $X$, i.e.
$$\Lambda (Y)=\Lambda(Y') \ \ \hbox{\rm whenever} Y \ \ \hbox{\rm and} \ \ Y' \ \ \hbox{\rm are bi-Lipschitz equivalent}.$$

\section{Auxiliary results}
Given a metric space $X=(X,d)$ and $c>0$, we denote the metric space $(X,cd)$ by $cX$.

\begin{lem}
For a metric space $(X,d)$,  $S\subset X$, and $c>0$, we have
$\lambda(cS,cX)=\lambda(S,X)$.
\end{lem}
\begin{proof}
Let $\varphi\in\lpf(cS)$ and $\|\varphi\|_{cS}=K$. Then $\varphi$ can also be
considered as an element of $\lpf(S)$ with $\|\varphi\|_{S}=K/c$. There exists an
extension $\bar\varphi\colon X\to\R$ of $\varphi$ with $\|\bar\varphi\|_{X}\le
(K\lambda(S,X))/c$. Considering $\bar\varphi$ as an element of $\lpf(cX)$, we see
that $\|\bar\varphi\|_{cX}\le (K\lambda(S,X))$. Therefore
$\lambda(cS,cX)\le\lambda(S,X)$. Arguing similarly, one can prove
also the opposite
inequality.
\end{proof}
\begin{lem}\label{l:2}
Let a metric pair $(S_1,X_1)$ be a retract of a metric pair $(S,X)$ under a
1-Lipschitz retraction. Then $\lambda(S_1,X_1)\le
\lambda(S,X)$.
\end{lem}
\begin{proof}
 Let $r\colon X\to X_1$ be a 1-Lipschitz retraction such that
 $r(S)=S_1$. Given a Lipschitz function $f\colon S_1\to\R$, we see
 that $f\circ(r|S)$ is a  Lipschitz function on $S$ with
 $\|f\circ(r|S)\|_S=\|f\|_{S_1}$. There is an extension $g\colon X\to\R$ of $f\circ(r|S)$
 with $\|g\|_X\le \lambda(S,X)\|f\circ(r|S)\|_S$. Then $g|X_1$  is an extension
 of $f$ over $X_1$ with $\|g|X_1\|_{X_1}\le \lambda(S,X)\|f\|_{S_1}$.
\end{proof}

\begin{prop}\label{p:1} Let $S$ be a closed subset of a compact metric space $X$ with
$|S|\ge2$. Then the following are equivalent:
\begin{enumerate}
  \item  there  exists a continuous linear extension operator
 $$\lpm(S) \to \lpm(X);$$  and

  \item  there  exists a continuous linear  extension operator
$$\lm(S) \to \lm(X).$$
\end{enumerate}

\end{prop}

\begin{proof} (1)$\Rightarrow$(2). Let $u\colon \lpm(S)\to \lpm(X)$ be
a continuous linear extension operator.

Note first that there exists $\tilde\varrho\in\lpm(X)$ such that
$\tilde\varrho(x,y)=0$ if and only if $(x,y)\in(S\times S)\cup
\Delta_X$ (by $\Delta_X $ we denote the diagonal of $X$). In order
to construct $\tilde\varrho$,
consider for any $x,y\in X\setminus S$ with $x\neq y$,
 the pseudometric $\varrho_{xy}$ on $S\cup\{x,y\}$, defined by
$$\varrho_{xy}|(S\times S)=0,\
\varrho_{xy}(x,y)=\varrho_{xy}(x,s)=\varrho_{xy}(y,s)=1$$ for any $s\in
S$. Denoting by $d$ the original metric on $X$ we see that
\begin{align*}
   & \varrho_{xy}(x,y)=1=(1/d(x,y))d(x,y), \\
   & \varrho_{xy}(x,s)=1\le(1/d(x,S))d(x,s), \\
   & \varrho_{xy}(y,s)=1\le(1/d(y,S))d(y,s)
\end{align*} for any $s\in S$. Therefore, $\varrho_{xy}$ is a
Lipschitz pseudometric with the Lipschitz constant
$$\max\{1/d(x,y), 1/d(x,S),
1/d(y,S)\}.$$
By a
result of Luukkainen \cite[Theorem 6.15]{L}, there exists a Lipschitz
pseudometric
$\tilde\varrho_{xy}$ on $X$ which is an extension of
$\varrho_{xy}$.

There exist (necessarily disjoint) neighborhoods $U_{xy}$ and $V_{xy}$ of $x$
and $y$ respectively,
such that $\tilde\varrho_{xy}(x',y')\neq0$, for every $x'\in
U_{xy}$ and $y'\in V_{xy}$. The family $$\{U_{xy}\times V_{xy}\mid x,y\in
((X\setminus S)\times (X\setminus S))\setminus\Delta_X\}$$ forms an open cover of
$((X\setminus S)\times (X\setminus S))\setminus\Delta_X$ and, by separability of
the latter set, there exists a sequence $(x_i,y_i)$ in $((X\setminus S)\times
(X\setminus S))\setminus\Delta_X$ such that $$\bigcup\{U_{x_iy_i}\times
V_{x_iy_i}\mid i\in\N\}=((X\setminus S)\times (X\setminus S))\setminus\Delta_X.$$
Let $$\displaystyle{\tilde\varrho=\sum_{i=1}^\infty
\frac{\tilde\varrho_{x_iy_i}}{2^i\|\tilde\varrho_{x_iy_i}\|_X}}.$$
Then, obviously, $\tilde\varrho\in\lpm(X)$ is as required.

Now define an operator $\tilde u\colon \lm(S)\to\lm(X)$ as follows. Let
$x_0,y_0\in S$, $x_0\neq y_0$. Let $\tilde
u(\delta)=u(\delta)+\delta(x_0,y_0)\tilde\varrho$, for any $\delta\in \lm(S)$. We
leave it
to the reader to easily verify that $\tilde u$ is
indeed a continuous
extension operator.

(2)$\Rightarrow$(1). We are going to show that $\lm(S)$ is dense in $\lpm(S)$.
Let $\varepsilon>0$. Given $\varrho\in\lpm(S)$, we see that
$\varrho_1=\varrho+\varepsilon d'$, where $d'$ is the original metric on $S$, is
an element of $\lm(S)$ with $\|\varrho-\varrho_1\|_S\le\varepsilon$.

Let $u\colon \lm(S)\to \lm(X)$ be a continuous linear extension operator. Since
$\lm(S)$ is dense in $\lpm(S)$ and the space $\lpm(X)$ is complete, there exists
a unique continuous extension
$\tilde u\colon
\lpm(S)\to \lpm(X)$
of $u$. Obviously, $\tilde u$ is a continuous
linear extension operator.

\end{proof}

\section{The Main result}

The following is the main result of this paper.

\begin{thm}\label{t:main}
There exists a closed subspace $A$ of a zero-dimensional compact metric space $X$
for which there is no extension operator for Lipschitz pseudometrics.
\end{thm}
\begin{proof}
We first
recall some results from \cite{BB}. Let $\Z_1^n(l)$ stand for
$\Z^n\cap[-l,l]^n$ endowed with the $\ell_1$-metric.
It was proved in
\cite[Lemma 10.5]{BB} that there exists
$c_1>0$ satisfying the following condition: For any natural $n$ there exists an integer  $l(n)>0$ and a subset $Y_n\subset
\Z_1^n(l(n))$ such that
\begin{equation}\label{f:neq}
 \lambda(Y_n,\Z_1^n(l(n)))\ge
c_1\sqrt{n}.
\end{equation}
 Let 
 $$X=\left(\coprod_{n=1}^\infty
\frac{1}{nl(n)}\Z_1^n(l(n))\right)\big/\sim,$$ 
where $\sim$ is the equivalence relation
which identifies all the origins, be the bouquet of the spaces $\Z^n_1(l(n))$,
$n\in\N$. 
We naturally identify every 
$$\frac{1}{nl(n)}\Z_1^n(l(n)), \  n\in\N,$$
with its copy,
which we denote by $X_n$, in $X$. The space $X$ is endowed with the maximal metric
$\varrho$ inducing the original metric on $X_n$, for every $n\in\N$. 
Obviously,
$X$ is a compact metric space. One can easily see that $X$ is zero-dimensional.

Let $$S_n=\frac{1}{nl(n)}Y_n\subset X_n,\ n\in\N.$$ Suppose that
there exists an extension operator $u\colon\lpm(S)\to\lpm(X)$,
where $$S=\left(\coprod_{n=1}^\infty S_n\right)/\sim\subset X.$$

By $\lpf(X,0)$ (resp. $\lpf(S,0)$) we denote the set of functions from $\lpf(X)$ (resp. $\lpf(S)$) that vanish at $0\in
X$ and by $\lpf^+(X,0)$ (resp. $\lpf^+(S,0)$) we denote the set of nonnegative functions from
$\lpf(X,0)$ (resp. $\lpf(S,0)$).

For any $x\in S\setminus\{0\}$, we denote by $f_x\colon S\to\R$ the function defined as follows:
$$f_x(x)=1 \ \ \hbox{\rm and} \ \ f_x(y)=0, \ \  \hbox{\rm whenever} \ \ y\neq x.$$
Then $f_x\in \lpf^+(S,0)$ and the collection $\{f_x\mid x\in S\setminus\{0\}\}$ forms a Schauder basis for $\lpf(S,0)$.

For any $x\in S\setminus\{0\}$, let $m(f_x)\colon S\times S\to \R$ be defined as follows: 
$$m(f_x)(y,z)=|f_x(y)-f_x(z)|.$$

Then, clearly $m(f_x)\in\lpm(S)$ and the assignment
$$f_x\mapsto m(f_x),  \ x\in S\setminus\{0\},$$
extends to a linear operator $m\colon \lpf^+(S,0)\to \lpm(S)$ by the formula
$$m\left(\sum_{x\in S\setminus\{0\}}\alpha_x f_x\right)=\sum_{x\in S\setminus\{0\}}\alpha_x m(f_x).$$

If $h\in \lpf^+(S,0)$, define $v(h)(y)=u(m(h))(y,0)$, for any $y\in X$. If $h\in \lpf(S,0)$, we represent $h$ as $h=h_1-h_2$, where $h_1,h_2\in \lpf^+(S,0)$, and  define $v(h)=v(h_1)-v(h_2)$. One can easily see that then $v(h)$ is well-defined.

 If $h\in\lpf(S)$, then $h-h(0)\in \lpf(S,0)$
and we put $v(h)=v(h-h(0))+h(0)$. By
a
direct verification we show that
$v\colon\lpf(S)\to\lpf(X)$ is a linear extension operator with $\|v\|<\infty$.
Therefore
$\lambda(S,X)<\infty$.

For every $n$, denote by $r_n\colon X\to X_n$ the retraction that sends the
complement of
$X_n$ to $0\in X_n$. Then evidently, $r_n$ is a 1-Lipschitz
retraction, $r_n(S)=S_n$ and, by Lemma \ref{l:2}, $\lambda(S_n, X_n)\le
\lambda(S,X)$. This obviously contradicts
the
 inequality (\ref{f:neq}).
\end{proof}

It follows by Proposition \ref{p:1} that for the subset $S$ of the space $X$
from the proof of Theorem \ref{t:main} there is no continuous linear operator
extending Lipschitz metrics.

\section{Epilogue}

One can follow  the proof of Theorem \ref{t:main} and conjecture that
$\lambda(S,X)<\infty$, whenever $\Lambda(S,X)<\infty$, for any subset $S$ of a
metric space $X$. Actually, the following question arises:

\begin{que}\label{q:1}
Are the numbers $\Lambda(S,X)$ and $\lambda(S,X)$ always equal?
\end{que}

We conjecture that the answer  is affirmative. However, if this is not the
case, then one can ask for a pseudometric analogue
of any result concerning linear
extensions of Lipschitz functions. As an example, we formulate the following
question inspired by results from \cite{bsh1}.

\begin{que}
 Let $(X,d)$ be a metric space and $\omega\colon\mathbb R_+\to\mathbb R_+$
 be a concave non-decreasing
function with $\omega(0)=0$. The function $d_{\omega}=\omega\circ d$ is a metric
on $X$. Are the properties $\Lambda(X,d_{\omega})<\infty$ and
$\Lambda(X,d)<\infty$ equivalent?
\end{que}

It was proved in \cite{tz} that there exists a linear operator which extends
partial pseudometrics with variable domain. The following question is, in some
sense, a strengthening of Question \ref{q:1}. Given a compact metric space $X$, we
let
$$\lpm=\bigcup\{\lpm(A)\mid A\text{ is a nonempty closed subset of
} X\}.$$ One can endow $\lpm$ with the following metric, $D$:
$$D(\varrho_1,\varrho_2)=\inf\{\|\tilde\varrho_1-\tilde\varrho_2\|_X
\mid \tilde\varrho_i\text{ is a Lipschitz extension of }\varrho_i\}.$$

\begin{que} Suppose that
for a metric space $X$, $\Lambda(X)<\infty$.
 Is there a continuous linear extension operator for partial
Lipschitz pseudometrics on $X$, i.e.
a map $u\colon
\lpm\to\lpm(X)$ which is continuous with respect to the metric $D$
and whose restriction onto every $\lpm(A)$, where $A$ is a subset of $X$, is
linear?
\end{que}

Note that the question which corresponds to the one above for the case of partial
Lipschitz fuctions is also open  - see \cite{ks} for the results on simultaneous
extensions of partial continuous functions.

\begin{que}
The space $\lpm(X)$ can also be endowed
with the topologies of the uniform and
pointwise convergence. Are there linear continuous extensions operators from
$\lpm(A)$ to $\lpm(X)$, where $A$ is a subset of $X$, which are continuous in
these topologies?
\end{que}

 We denote the set of all Lipschitz ultrametrics on a subset $Y$ of a
zero-dimensional metric space
by $\lpu(Y)$. It was proved in \cite{tz1} (see also \cite{s})
that there exists a continuous extension operator that extends ultrametrics
defined on a closed subspace of a zero-dimensional compact metric space and
preserves the operation $\max$.

\begin{que}
Given a subset $A$ of a zero-dimensional metric space $X$, is there a continuous
extension operator for Lipschitz ultrapseudometrics
$u\colon
\lpu(A)\to\lpu(X)$
which
preserves the operation $\max$ (is
homogeneous)?
\end{que}


\end{document}